\documentclass[article,12pt]{amsart}

\usepackage{amsfonts}
\usepackage{amsmath}
\usepackage{mathrsfs}
\usepackage{graphicx}
\usepackage{color}
\usepackage{epsfig}
\usepackage{subcaption}
\usepackage{enumerate}
\usepackage{comment}

 \usepackage[foot]{amsaddr}

\numberwithin{equation}{section}
\theoremstyle{plain}
\newtheorem{thm}{Theorem}[section]
\newtheorem{rem}{Remark}[section]
\newtheorem{prop}{Proposition}[section]
\newtheorem{cor}{Corollary}[section]
\newtheorem{lem}{Lemma}[section]

\newcommand{\dE}{\mathbb{E}}
\newcommand{\dR}{\mathbb{R}}

\newcommand{\dP}{\mathbb{P}}

\newcommand{\dN}{\mathbb{N}}
\newcommand{\dZ}{\mathbb{Z}}

\newcommand{\cB}{\mathcal{B}}
\newcommand{\cC}{\mathcal{C}}

\newcommand{\cL}{\mathcal{L}}

\newcommand{\cN}{\mathcal{N}}

\newcommand{\cS}{\mathcal{S}}
\newcommand{\rI}{\mathrm{I}}
\newcommand{\cF}{\mathcal{F}}

\newcommand{\ind}{\mbox{1}\kern-.25em \mbox{I}}

\def\build#1_#2^#3{\mathrel{\mathop{\kern 0pt#1}\limits_{#2}^{#3}}}

\def\videbox{\mathbin{\vbox{\hrule\hbox{\vrule height1.4ex \kern.6em\vrule height1.4ex}\hrule}}}
\def\demend{\hfill $\videbox$\\}

\newcommand{\Hm}[1]{\leavevmode{\marginpar{\tiny%
$\hbox to 0mm{\hspace*{-0.5mm}$\leftarrow$\hss}%
\vcenter{\vrule depth 0.1mm height 0.1mm width \the\marginparwidth}%
\hbox to 0mm{\hss$\rightarrow$\hspace*{-0.5mm}}$\\\relax\raggedright
#1}}}
\usepackage{color}

\begin{document}
\title[Descents and inverse descents in a random permutation]
{Sharp analysis on the joint distribution of the number of descents and inverse descents in a random permutation}
\author{Bernard Bercu}
\thanks{The corresponding author is Michel Bonnefont,  michel.bonnefont@math.u-bordeaux.fr}
\address{Universit\'e de Bordeaux, Institut de Math\'ematiques de Bordeaux,
UMR CNRS 5251, 351 Cours de la Lib\'eration, 33405 Talence cedex, France.}
\email{bernard.bercu@math.u-bordeaux.fr}
\author{Michel Bonnefont}
\email{michel.bonnefont@math.u-bordeaux.fr}
\author{Luis Fredes}
\email{luis.fredes@math.u-bordeaux.fr}
\author{Adrien Richou}
\email{adrien.richou@math.u-bordeaux.fr}
\date{\today}

\begin{abstract}
Chatteerjee and Diaconis have recently shown the asymptotic normality for the joint distribution of the number of descents and inverse descents in a random permutation. A noteworthy point of their results is that the asymptotic variance of the normal distribution is diagonal, which means that the number of descents and inverse descents are asymptotically uncorrelated.
The goal of this paper is to go further in this analysis by proving a large deviation principle
for the joint distribution. We shall show that the rate function of the joint distribution
is the sum of the rate functions of the marginal distributions, which also means that the number of descents and inverse descents are asymptotically independent at the large deviation level. However,
we are going to prove that they are finely dependent at the sharp large deviation level.
\end{abstract}


\maketitle

\vspace{1ex}


\section{Introduction}
\label{S-I}


Let $\cS_n$ be the symmetric group of permutations on the set of integers $\{1,\ldots,n\}$ 
where $n \geq 1$. A permutation $\pi_n \in \cS_n$ has a descent at position $k \in \{1,\ldots,n-1\}$ if  $\pi_n(k)>\pi_n(k+1)$. 
Let $D_n=D_n(\pi_n)$ be the random variable counting the number of descents 
of a permutation $\pi_n$ chosen uniformly at random from $\cS_n$ and denote by 
$D^\prime_n=D_n(\pi_n^{-1})$ the number of descents of the inverse
$\pi_n^{-1}$ of $\pi_n$. We clearly have for all $n \geq 2$,
\begin{equation}
\label{DEFDN}
D_{n}=\sum_{k=1}^{n-1} \rI_{\{\pi_n(k)>\pi_n(k+1)\}}
\hspace{1cm}\text{and}\hspace{1cm}
D^\prime_{n}=\sum_{k=1}^{n-1} \rI_{\{\pi_n^{-1}(k)>\pi_n^{-1}(k+1)\}}.
\end{equation}
A wide range of literature is available on the asymptotic behavior of the marginal distribution 
$(D_n)$. It is well-known
\cite{Freedman1965}, \cite{Garet2021}, \cite{Ozdemir2021} that
\begin{equation}
\label{ASCVGDN}
\lim_{n \rightarrow \infty}\frac{D_n}{n}=\frac{1}{2} \hspace{1cm} \text{a.s.}
\end{equation}
and
\begin{equation}
\label{ANDN}
\sqrt{n}\left(\frac{D_n}{n} - \frac{1}{2}\right)
\underset{n\rightarrow+\infty}{\overset{\cL}{\longrightarrow}}
 \cN \Bigl(0, \frac{1}{12}\Bigr).
\end{equation}
Moreover, Bryc, Minda and Sethuraman \cite{Bryc2009} have shown that
$(D_n/n)$ satisfies a large deviation principle (LDP) with good rate 
function given, for all $x \in \dR$, by
\begin{equation}
\label{DEFI}
I(x)=\sup_{ t \in \dR} \bigl\{xt - L(t) \bigr\}
\end{equation}
where 
\begin{equation}\label{DEFL}
L(t)=\log \left(\frac{\exp(t)-1}{t}\right).
\end{equation}
Surprisingly, to the best of our knowledge, very few result are available on the 
asymptotic behavior of the joint distribution $(D_n,D^\prime_n)$, except
the recent contribution of Chatteerjee and Diaconis \cite{Chatterjee2017}.
By using a normal approximation method due to Chatterjee
\cite{Chatterjee2008}, it is shown in \cite{Chatterjee2017} that
\begin{equation}
\label{CANDN}
\sqrt{n}
\left(
\frac{D_n}{n} - \frac{1}{2},
\frac{D_n^\prime}{n} - \frac{1}{2}
\right)
\underset{n\rightarrow+\infty}{\overset{\cL}{\longrightarrow}}
\cN \bigl(0, \Gamma\bigr),
\end{equation}
where $\Gamma$ stands for the diagonal matrix
$$
\Gamma=\frac{1}{12} 
\begin{pmatrix}
1 & 0\\
0 & 1
\end{pmatrix}.
$$
It clearly means that $D_n$ and $D^\prime_n$ are asymptotically uncorrelated. 
One can observe that the almost sure convergence of the couple $(D_n/n, D'_n/n)$ to $(1/2,1/2)$ can be obtained from Remark 3.2 in \cite{Bercu2024}.
As suggested in Section 3 of \cite{Chatterjee2017}, it is possible to go deeper into the analysis of the joint distribution $(D_n,D^\prime_n)$ by proving that the couple $(D_n/n,D^\prime_n/n)$ satisfies
an LDP with good rate function $I$ defined, for all $x,y \in \dR$, by
\begin{equation}
\label{DEFCI}
I(x,y)=I(x)+I(y).
\end{equation}
It also implies that, at the large deviation level, $D_n$ and $D^\prime_n$ are asymptotically independent. Moreever, it was recently proven in \cite{Bercu2024} that the sequence $(D_n)$ satisfies a sharp large deviation principle (SLDP). We are going to improve this result by showing that the couple $(D_n,D^\prime_n)$ satisfies a SLDP which will highlight the very
fine dependence between $D_n$ and $D^\prime_n$ at the sharp large deviation level.
\ \vspace{1ex} \\
The paper is organized as follows. Section \ref{S-PS} deals with an explicit construction of the probability space on which we will be working.
Section \ref{S-MG} is devoted to our martingale approach which allows us to find again a direct proof of the asymptotic normality \eqref{CANDN} and to propose new standard results for the
joint distribution $(D_n,D^\prime_n)$ such as a functional central limit theorem.
The main results of the paper are given in Section \ref{S-MR} where we establish the large
deviation properties of the couple $(D_n,D^\prime_n)$. 
A short conclusion is given in Section \ref{S-C}. All technical proofs are postponed to Sections \ref{S-PLDP} to \ref{S-PMG}.


\section{The probability space}
\label{S-PS}
Before investigating the joint distribution of $(D_n,D^\prime_n)$, it is necessary to provide an explicit construction of the probability space on which we will be working. For that purpose,
we are going to study how a permutation $\pi_{n+1}$ of size $n+1$ adds descents when it is created from a permutation $\pi_n$ of size $n$. We will do this by means of a graphical representation as explained below. 
We start by presenting a permutation $\pi_n$ as the set of all points $GR(\pi_n)$ 
inside $[0,1]^2$ 
where 
$$\pi_n(i)=j\implies \Bigl(\frac{i}{n+1},\frac{j}{n+1}\Bigr) \in GR(\pi_n).$$
One can observe that the image of this transformation has the property of having exactly one point in each line and column of type $i/(n+1)$ for $i\in \{1,\dots,n\}$ and to send all the points inside $[0,1]^2$. We call $\mathcal{GR}_n$ the set of all points with this property. 
It is easy to see that the sets $\cS_n$ and $\mathcal{GR}_n$ are in bijection 
since the points in each element of $\mathcal{GR}_n$ defines the map of a bijective function.
We define the cell $\cC_{n}(i,j)$ to be the region inside $[0,1]^2$ defined, for all
$(i,j)\in \{1,\dots,n+1\}^2$, by 
$$
\cC_{n}(i,j)=\Bigl]\frac{i-1}{n+1},\frac{i}{n+1}\Bigr[\times \Bigl]\frac{j-1}{n+1},\frac{j}{n+1}\Bigr[.
$$ 

For a point $u \in [0,1]^2$, we construct $\pi_{n+1}$ as a function of $u$ and $\pi_n$ as described in Figure \ref{fig:graph_rep_0} below,
$$
	\pi_{n+1}(i) = \begin{cases}
		\pi_{n}(i)&\text{ if } u \in \cC_{n}(k,\ell),\; i < k \text{ and } \; \pi_n(i)< \ell \\
		\pi_{n}(i)+1&\text{ if } u  \in \cC_{n}(k,\ell), \; i < k \text{ and } \; \pi_n(i)\geq  \ell \\
		\pi_{n}(i-1)&\text{ if } u \in \cC_{n}(k,\ell), \; i>  k \text{ and } \; \pi_n(i)< \ell \\
		\pi_{n}(i-1)+1&\text{ if } u\in \cC_{n}(k,\ell), \; i>  k \text{ and } \; \pi_n(i)\geq \ell\\ 
		\ell &\text{ if } u \in \cC_{n}(k,\ell), \; i=  k
	\end{cases}
$$
\begin{figure}[h!]
	\centering
	\includegraphics[scale=0.6]{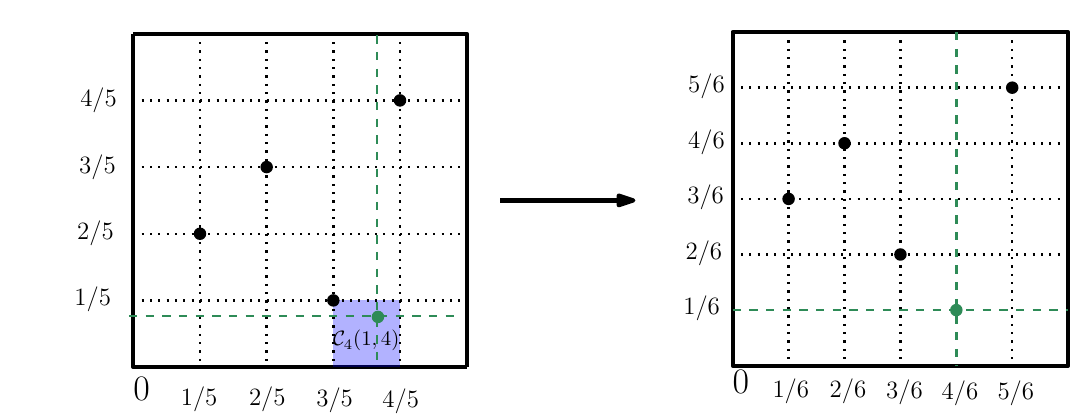}
	\caption{Graphical representation of $\pi_4 = (2314)$ to the left. We colored in light blue the cell $\cC_4(1,4)$. We represent to the right-side the resulting permutation $\pi_{5}= (34215)$ when $u\in \cC_4(1,4)$. We present the point $u$ and vertical and horizontal lines crossing it in green in order to visually see the transformation.}
	\label{fig:graph_rep_0}
\end{figure}

\noindent
It is clear that our construction of $\pi_{n+1}$ only depends on $\pi_n$ and on the choice of the indexes $(k,\ell)$ of the cell where the point $u$ lands. One can think of this transformation as the result of adding, in the graphic representation of $\pi_n$, new horizontal and vertical lines crossing the point $u$ and after rescaling these lines in order to keep the relative order of points with spacing $1/(n+2)$, we obtain the graphical representation of $\pi_{n+1}$. 
\ \vspace{1ex} \\
Up to now, $\pi_{n+1}$ is built in a deterministic way. In what follows, what will play the role of $u$ is a random variable $U_{n+1}$ uniformly chosen in $[0,1]^2$ and independent of the permutation $\pi_n$. Since the boundary lines of cells have Lebesgue measure 0, we suppose without loss of generality that $U_{n+1}$ does not belong to the boundary of any cell. It is simple to prove that if $\pi_{n}$ is chosen uniformly at random from $\cS_n$, then $\pi_{n+1}$ is a uniform permutation over $\cS_{n+1}$. From this construction, we can built the sequence $(D_n,D^\prime_n)$ on a unique probability space only depending on the whole sequence $(U_n)$.
\ \vspace{1ex} \\
First of all, we start at the origin $(D_1,D^\prime_1)=(0,0)$. Then, for all $n \geq 1$, we denote by
$\Delta D_{n+1}$ and $\Delta D^\prime_{n+1}$ the increments $\Delta D_{n+1}=D_{n+1}-D_n$
and $\Delta D^\prime_{n+1}= D^\prime_{n+1} - D^\prime_{n}$.
Thanks to our representation, we can study the number of cells $\cC_{n}(k,\ell)$ that generate prescribed increments $\Delta D_{n+1}=a$ and $\Delta D'_{n+1}=b$, $a,b\in \{0,1\}$ for the permutation $\pi_n$ and its inverse $\pi_n^{-1}$ when $u\in \cC_{n}(k,\ell)$.
For $a,b\in \{0,1\}$, let
\begin{equation*}
    N_{n+1}(a,b, \pi_n)= \text{Card} \big\{(k,\ell) \in \{1, \dots,n+1\}^2 \big\slash
    \Delta D_{n+1}= a, \Delta D'_{n+1}= b, u\in \cC_{n}(k,\ell)\big\}.
\end{equation*}

\noindent
Our first result shows that $N_{n+1}(a,b, \pi_n)$ depends only on $D_n$ and $D^\prime_n$ as follows.

\begin{thm}
\label{T-PROBASPACE}
For every permutation $\pi_{n+1}$ chosen uniformly at random from $\cS_{n+1}$, we have
\begin{equation*}
   N_{n+1} (a,b,\pi_n) =\left \{ \begin{array}{ccc}
     (n-D_n)(n-D^\prime_n)+n & \text{ if } & (a,b)=(1,1), \vspace{1ex}\\
     (n-D_n)(D'_n+1)-n & \text{ if }  & (a,b)=(1,0), \vspace{1ex}\\
      (D_n+1)(n-D'_n)-n  & \text{ if } & (a,b)=(0,1), \vspace{1ex}\\
     (D_n+1)(D'_n+1)+n & \text{ if }  & (a,b)=(0,0).
   \end{array} \nonumber \right.
\end{equation*}
\end{thm}

\begin{proof}
The proof is given in Section \ref{S-PPS}.
\end{proof}


\section{Our martingale approach}
\label{S-MG}
Let $V_n$ be the two-dimensional random vector $V_n=(D_n, D^\prime_n)$.
Denote by $\cF_n$ the $\sigma$-algebra 
$\cF_n=\sigma(D_1, \ldots,D_n,D^\prime_1, \ldots,D^\prime_n)$.
It  follows from Theorem \ref{T-PROBASPACE} that for all $n \geq 1$,
\begin{equation}
\label{DECVN}
V_{n+1}=V_n+\xi_{n+1}
\end{equation}
where
\begin{equation*}
   \dP(\xi_{n+1}  = (a,b)|\cF_n) =\left \{ \begin{array}{ccc}
    {\displaystyle \frac{(n-D_n)(n-D'_n)+n}{(n+1)^2}}  & \text{ if } & (a,b)=(1,1), \vspace{1ex}\\
     {\displaystyle \frac{(n-D_n)(D'_n+1)-n}{(n+1)^2} } & \text{ if }  & (a,b)=(1,0), \vspace{1ex}\\
      {\displaystyle \frac{(D_n+1)(n-D'_n)-n}{(n+1)^2}}  & \text{ if } & (a,b)=(0,1), \vspace{1ex}\\
     {\displaystyle \frac{(D_n+1)(D'_n+1)+n}{(n+1)^2} } & \text{ if }  & (a,b)=(0,0).
   \end{array} \nonumber \right.
\end{equation*}
It implies that the conditional distribution of the marginals of $\xi_{n+1}$ given $\cF_n$ are the correlated Bernoulli $\cB(p_n)$ and $\cB(p^\prime_n)$ distributions where
\begin{equation}
\label{DEFPN}
p_n=\frac{n-D_n}{n+1} \hspace{1cm}\text{and} \hspace{1cm} p_n^\prime=\frac{n-D_n^\prime}{n+1}.
\end{equation}
Moreover, we also have
$$
\dE[\xi_{n+1}\xi_{n+1}^T | \cF_n]=\begin{pmatrix}
p_n & p_np_n^\prime +r_n\\
p_np_n^\prime +r_n & p_n^\prime
\end{pmatrix}
$$
where
$$
r_n=\frac{n}{(n+1)^2}.
$$
Consequently, we obtain from \eqref{DECVN} that
\begin{equation}
\label{CMVN}
\dE[V_{n+1} | \cF_n] = \dE[V_n+ \xi_{n+1} | \cF_n]
= \begin{pmatrix}
D_n + p_n\\
D_n^\prime + p_n^\prime
\end{pmatrix}.
\end{equation}
Let $(M_n)$ be the sequence defined for all $n \geq 1$ by
\begin{equation}
\label{DEFMN}
M_{n} = n \left(V_n
- \frac{(n-1)}{2}v \right)
\end{equation}
where $v^T=(1,1)$.
It follows from \eqref{CMVN} that $(M_n)$ is a locally square integrable martingale.
One can easily see that its predictable quadratic variation reduces to 
\begin{eqnarray}
\label{IPMN}
\langle M \rangle_n  &=& \sum_{k=1}^{n-1} \dE[(M_{k+1}-M_k)(M_{k+1}-M_k)^T|\cF_{k}], \nonumber \\
&=&\sum_{k=1}^{n-1} \begin{pmatrix}
(k-D_k)(D_k+1) & k\\
k & (k-D_k^\prime)(D_k^\prime+1)
\end{pmatrix}.
\end{eqnarray}
Our martingale decomposition \eqref{DEFMN} allows us to propose a straightforward proof
of the asymptotic normality \eqref{CANDN} previously established via a more sophisticated
approach by Chatteerjee and Diaconis \cite{Chatterjee2017}. We can also establish new
asymptotic results such as a functional central limit theorem and a law of iterated logarithm as
follows.

\begin{prop}
\label{P-FCLT}
Let $D([0,\infty[)$ be the Skorokhod space of right-continuous functions with left-hand limits.
We have the distributional convergence in $D([0,+\infty[)$,
\begin{equation}
\label{FCLTDN}
\left(\sqrt{n}
\Bigl(
\frac{D_{\lfloor nt \rfloor}}{\lfloor nt\rfloor} - \frac{1}{2},
\frac{D_{\lfloor nt\rfloor}^\prime}{\lfloor nt\rfloor} - \frac{1}{2}
\Bigr), \ t \geq 0 \right)
\Longrightarrow (W_t, \ t \geq 0)
\end{equation}
where $(W_t)$ is a continuous two-dimensional centered Gaussian process starting from the origin with covariance
matrix given, for $0< s \leq t$, by
$$
\dE[ W_t W_s^T]= \frac{s}{12 t^2}
\begin{pmatrix}
1 & 0\\
0 & 1
\end{pmatrix}.
$$ 
In particular, we find again the asymptotic normality \eqref{CANDN}.
\end{prop}

\begin{prop}
\label{P-QSLLIL}
We have the quadratic strong law
\begin{equation}
\lim_{n\rightarrow \infty} 
\frac{1}{\log n} \sum_{k=1}^n 
\begin{pmatrix}
\frac{D_k}{k} - \frac{1}{2} \\
\frac{D_k^\prime}{k} - \frac{1}{2}
\end{pmatrix}
\begin{pmatrix}
\frac{D_k}{k} - \frac{1}{2} \\
\frac{D_k^\prime}{k} - \frac{1}{2}
\end{pmatrix}^{\!T}
=\frac{1}{12} \begin{pmatrix}
1 & 0\\
0 & 1
\end{pmatrix} \hspace{1cm} \text{a.s.}
\label{QSLDN}
\end{equation}
In particular,
\begin{equation}
\lim_{n\rightarrow \infty} 
\frac{1}{\log n} \sum_{k=1}^n 
\Bigl(\frac{D_k}{k} - \frac{1}{2} \Bigr)^2+
\Bigl(\frac{D_k^\prime}{k} - \frac{1}{2} \Bigr)^2
=\frac{1}{6}  \hspace{1cm} \text{a.s.}
\label{QSLTRACEDN}
\end{equation}
Moreover, we also have the law of iterated logarithm
\begin{equation}
 \limsup_{n \rightarrow \infty} \frac{n}{ 2\log \log n}
 \left( \Bigl(\frac{D_n}{n} - \frac{1}{2} \Bigr)^2+
\Bigl(\frac{D_n^\prime}{n} - \frac{1}{2} \Bigr)^2 \right)
= \frac{1}{6}\hspace{1cm} \text{a.s.}
\label{LILTRACEDN}
\end{equation}
 \end{prop}

\begin{proof}
The proofs are postponed to Section \ref{S-PMG}
\end{proof}


\section{Main results}
\label{S-MR}

Our first result is the LDP for the couple $(D_n/n,D^\prime_n/n)$ which extends the
LDP for $(D_n/n)$ previously established by Bryc, Minda and Sethuraman
\cite{Bryc2009}.

\begin{thm}
\label{T-LDP}
The couple $(D_n/n,D'_n/n)$ satisfies an LDP with good rate function $I$ given by 
\eqref{DEFCI}.
\end{thm}

Our second result is the SLDP for the sequence $(D_n,D'_n)$ which nicely improves the SLDP 
recently established for $(D_n)$ in Theorem 3.1 of \cite{Bercu2024}. For any positive real number $x$ and $n \in \mathbb{N}^*$, 
denote $\{x_n\}=nx-\lceil (n-1)x \rceil$. Let us remark that $\{x_n\} \in [x-1,x]$ for any $n \in \mathbb{N}^*$. 
Instead of  dividing $(D_n,D^\prime_n)$ by $n$,  it is more natural to divide them by $n-1$ for the following reason.
Let $A_n$ be the number of ascents of $\pi_n \in \cS_n$. Then, it is clear that 
 $D_n+A_n=n-1$. Moreover, $D_n$ and $A_n$ share the same distribution. Consequently, $D_n/(n-1)$ is symmetric 
 with respect to $1/2$. The same applies for $D'_n$.

\begin{thm}
\label{T-SLDP}
For any $x$ and $y$ in $]1/2,1[$, we have
\begin{equation}
\label{SLDPR}
\dP \Big(\frac{D_n}{n-1} \geq x, \frac{D'_n}{n-1} \geq y  \Big) = 
\frac{\exp(-nI(x,y)+\varphi_n(x,y))}{2\pi n \sigma_x t_x \sigma_y t_y }
\big[ 1+o(1) \big]
\end{equation} 
where the value $t_x$ is the unique solution of $L^\prime(t_x)=x$, $\sigma_x^2=L^{\prime \prime}(t_x)$ and 
\begin{equation}
\label{COFACTPLUS}
\varphi_n(x,y)=\{x_n\}t_x+\{y_n\}t_y+\frac{1}{2}t_x t_y.
\end{equation}
In the same spirit, we have for any $x$ and $y$ in $]1/2,1[$,
\begin{equation}
\label{SLDPR2}
\dP \Big(\frac{D_n}{n-1} \leq 1-x, \frac{D'_n}{n-1} \leq 1-y  \Big) = 
\frac{\exp(-nI(x,y)+\varphi_n(x,y))}{2\pi n \sigma_x t_x \sigma_y t_y }
\big[ 1+o(1) \big].
\end{equation}
Moreover, we also have for any $x$ and $y$ in $]1/2,1[$,
\begin{equation}
\label{SLDPR3}
\dP \Big(\frac{D_n}{n-1} \leq 1-x, \frac{D'_n}{n-1} \geq y  \Big) = 
\frac{\exp(-nI(x,y)+\phi_n(x,y))}{2\pi n \sigma_x t_x \sigma_y t_y }
\big[ 1+o(1) \big],
\end{equation}
and
\begin{equation}
\label{SLDPR4}
\dP \Big(\frac{D_n}{n-1} \geq x, \frac{D'_n}{n-1} \leq 1-y  \Big) =
\frac{\exp(-nI(x,y)+\phi_n(x,y))}{2\pi n \sigma_x t_x \sigma_y t_y }
\big[ 1+o(1) \big] 
\end{equation}
where
\begin{equation}
\label{COFACTMINUS}
\phi_n(x,y)=\{x_n\}t_x+\{y_n\}t_y-\frac{1}{2}t_x t_y.
\end{equation}
\end{thm}

\noindent
By using a direct modification in the proof of Theorem 3.1 in \cite{Bercu2024}, we obtain that for any $x$ in $]1/2,1[$,
\begin{equation*}
\dP \Big(\frac{D_n}{n-1} \geq x\Big)=\dP \Big(\frac{D'_n}{n-1} \geq x\Big) = \frac{\exp(-nI(x)+\{x_n\}t_x)}{\sigma_x t_x \sqrt{2\pi n}}
\big[ 1+o(1) \big].
\end{equation*}
Moreover, we have by symmetry that
\begin{equation*}
    \dP \Big(\frac{D_n}{n-1} \geq x\Big)=\dP \Big(\frac{D_n}{n-1} \leq 1-x\Big).
\end{equation*}
Therefore, an immediately consequence of Theorem \ref{T-SLDP} is as follows.
\begin{cor} 
\label{C-DEPSLDP}
For all $x,y \in ]1/2,1[$, we have
\begin{align*}
    \dP \Big(\frac{D_n}{n-1} \geq x, \frac{D'_n}{n-1} \geq y  \Big) &= \dP \Big(\frac{D_n}{n-1} \geq x \Big) \dP \Big( \frac{D'_n}{n-1} \geq y  \Big) \exp\left( \frac{t_x t_y}{2}\right)
\big[ 1+o(1) \big],\\
\dP \Big(\frac{D_n}{n-1} \leq 1-x, \frac{D'_n}{n-1} \leq 1-y  \Big) &= \dP \Big(\frac{D_n}{n-1} \leq 1-x \Big) \dP \Big( \frac{D'_n}{n-1} \leq 1-y  \Big) \exp\left( \frac{t_x t_y}{2}\right)
\big[ 1+o(1) \big],\\
\dP \Big(\frac{D_n}{n-1} \leq 1-x, \frac{D'_n}{n-1} \geq y  \Big) &= \dP \Big(\frac{D_n}{n-1} \leq 1-x \Big) \dP \Big( \frac{D'_n}{n-1} \geq y  \Big) \exp\left( -\frac{t_x t_y}{2}\right)
\big[ 1+o(1) \big],\\
\dP \Big(\frac{D_n}{n-1} \geq x, \frac{D'_n}{n-1} \leq 1-y  \Big) &= \dP \Big(\frac{D_n}{n-1} \geq x \Big) \dP \Big( \frac{D'_n}{n-1} \leq 1-y  \Big) \exp\left( -\frac{t_x t_y}{2}\right)
\big[ 1+o(1) \big].
\end{align*}
\end{cor}

\begin{rem}
Corollary \ref{C-DEPSLDP} highlights the very
fine dependence between $D_n$ and $D^\prime_n$ at the sharp large deviation level.
\end{rem}

\noindent
Chatteerjee and Diaconis found challenging to study the asymptotic behavior of
$$
T_n=D_n+D^\prime_n.
$$
They proved in Theorem 1.1 of \cite{Chatterjee2017} that
\begin{equation}
\label{ANT}
\sqrt{n}
\left(
\frac{T_n}{n} - 1\right)
\underset{n\rightarrow+\infty}{\overset{\cL}{\longrightarrow}}
\cN \Bigl(0, \frac{1}{6}\Bigr).
\end{equation}
We immediately deduce from Theorem \ref{T-LDP} together with the contraction principle
\cite{Dembo1998} that $(T_n/n)$ satisfies an LDP as follows.
\begin{cor} 
The sequence $(T_n/n)$ satisfies an LDP with good rate function $J$ defined, for all $y \in \dR$, by
\begin{equation}
\label{DEFJ}
J(y) = \inf \Bigl\{ I(x)+I(y-x), \ x\in \dR \Bigr\}.
\end{equation} 
\end{cor}

\section{Conclusion}
\label{S-C}

Our strategy of proof may be extended to some of other
well-known statistics on random permutations. More precisely, if we are interested
in the major index \cite{MacMahon1913}, 
a central limit theorem for the joint distribution of the major
index and the major index of the inverse has been obtained in 
\cite{Baxter2010, Swanson2022}. However, the LDP
is available only for the marginal distribution \cite{Meliot2022}. The LDP for the joint distribution,
or even for the quadruplet distribution adding the couple $(D_n, D'_n)$ is still an open problem
and may be tackled using some old results concerning their Laplace transform
\cite{Garsia1979, Roselle1974}, see also the link with inversions in \cite{Foata1978}. These questions are left for future work.

\section{Proofs of the large deviation results}
\label{S-PLDP}

\subsection{Some preliminary results}

Denote by $m_n$ the Laplace transform of the couple $(D_n,D'_n)$ defined, for all 
$t,s \in \dR$, by 
\begin{equation*}
    m_n(t,s)  = \dE \bigl[\exp(t D_n + s D'_n) \bigr].
\end{equation*}
One can notice that $m_n(t,s)$ is always finite for all $t,s \in \mathbb{R}$ since $D_n$ and $D_n'$ are bounded by $n$. We shall make use of the
unsigned Stirling numbers of the first kind, which are defined algebraically as the coefficients of the rising factorial via the identity given for all $x \in \dR$ and $n \in \dN$,
\begin{equation}
\label{RISEFACT}
    (x)^{(n)} =x(x+1) \cdots(x+n-1)= \sum_{k=0}^n {n \brack k} x^k.
\end{equation}

\begin{lem}
\label{LEM-mn}
For all $t,s \in \mathbb{R}^*$, we have
\begin{equation}
\label{EQ-LEM-mn}
m_n(t,s) = \left(\frac{e^t-1}{t}\right)^n\left(\frac{e^s-1}{s}\right)^n\left(\frac{1-e^{-t}}{t}\right)\left(\frac{1-e^{-s}}{s}\right)S_n^{1}(t,s) 
\end{equation}
where for $a \in \{0,1\}$,
\begin{equation*}
    S_n^{a}(t,s) = \sum_{k=0}^{n-1} { n\brack n-k} \left(\frac{(n-k)!}{n!}\right)^2 (st)^k
    \bigl(1+a r_{n-k}(t)\bigr)\bigl(1+a r_{n-k}(s)\bigr)
\end{equation*}
and 
\begin{equation*}
    r_{n}(t) = \sum_{\ell \in \mathbb{Z}^*} \left(1+\frac{2i\pi \ell}{t}\right)^{-(n+1)}.
\end{equation*}
\end{lem}

\noindent
Let us remark that the notation $S_n^0$ is not useful for the moment but it will be used in the proof of Theorem \ref{T-SLDP}.

\begin{proof}
It follows from Section 7 in \cite{Carlitz1966}, see also Theorem 2 in \cite{Petersen2013} or more recently Section 3 in \cite{Chatterjee2017}, that for all $p,q \in (-1,1)$,
\begin{equation}
\label{EQ-FCTGENE}
 \mathbb{E}\big[p^{D_n}q^{D'_n}\big] = \frac{(1-p)^{n+1}(1-q)^{n+1}}{pq n!}\sum_{k,\ell \geq 0} \binom{k\ell+n-1}{n} p^k q^\ell. 
\end{equation}
Let us start by assuming that $t<0$ and $s<0$. 
By using \eqref{EQ-FCTGENE} with $p = e^{t}<1$ and $q=e^{s}<1$, we obtain that
\begin{eqnarray*}
m_n(t,s) &=& \frac{(1-e^{t})^{n+1}(1-e^{s})^{n+1}}{e^{t}e^{s} n!}\sum_{k,\ell \geq 0} \binom{k\ell+n-1}{n} e^{tk} e^{s\ell},\\ 
&=& \frac{(e^t-1)^{n+1}(e^s-1)^{n+1}}{e^{t}e^{s} (n!)^2}\sum_{k,\ell \geq 0} (k\ell)^{(n)} e^{tk} e^{s\ell}
\end{eqnarray*}
where the rising factorial $(kl)^{(n)}$ was previously defined in \eqref{RISEFACT}.  Hence, we obtain
from \eqref{RISEFACT} that for all $n \geq 1$,
\begin{align*}
   m_n(t,s) &= \frac{(e^t-1)^{n+1}(e^s-1)^{n+1}}{e^{t}e^{s} (n!)^2} \sum_{j=1}^n {n \brack j} \sum_{k \geq 0} k^j e^{t k} \sum_{\ell \geq 0} l^j e^{s \ell}.
\end{align*}
It follows from the series representation of the polylogarithm function given e.g. by formula (13.1)
in \cite{Gradshteyn2015, Wood1992}
that for all $n \geq 1$,
\begin{align*}
   m_n(t,s) =& \frac{(e^t-1)^{n+1}(e^s-1)^{n+1}}{e^{t}e^{s} (n!)^2} \sum_{j=1}^n {n \brack j} \frac{(j!)^2}{(ts)^{j+1}}\sum_{k \in \mathbb{Z}} \frac{1}{\left(1+\frac{2i\pi k}{t}\right)^{j+1}}
   \sum_{\ell \in \mathbb{Z}} \frac{1}{\left(1+\frac{2i\pi \ell}{s}\right)^{j+1}}\\
   =& \frac{(e^t-1)^{n+1}(e^s-1)^{n+1}}{e^{t}e^{s}(st)^{n+1}} \sum_{j=0}^{n-1} {n \brack n-j} \left(\frac{(n-j)!}{n!}\right)^2 (st)^j R_{n-j}(t,s)
\end{align*}
where
$$
R_{n-j}(t,s)=\bigl(1+ r_{n-j}(t)\bigr)\bigl(1+r_{n-j}(s)\bigr),
$$
which is exactly what we wanted to prove.
From now on, we can remark that $m_n$ is clearly an holomorphic function with respect to $t$, resp. $s$, on $\mathbb{C}$. Moreover, the right hand-side of \eqref{EQ-LEM-mn} is also an analytical function with respect to $t$, resp. $s$, at least on $\mathbb{C}\setminus 2i\pi \mathbb{Z}$, which means that \eqref{EQ-LEM-mn} holds true on $(\mathbb{R}^*)^2$.
\end{proof}

\begin{rem}
\label{REM-LAPLACE-UNIDIM}
One can observe that for all $ t \in \mathbb{R}^*$,
$$m_n(t,0)=m_n(0,t)=\left(\frac{e^t-1}{t}\right)^n \left(\frac{1-e^{-t}}{t}\right)\sum_{\ell \in \mathbb{Z}} \left(1+\frac{2i\pi \ell}{t}\right)^{-(n+1)}.$$
This formula can be used to simplify several computations done in \cite{Bercu2024}.
\end{rem}

\noindent
We have seen in the proof of Lemma \ref{LEM-mn} that for all $n \geq 1$, $r_n$  is an analytical function on $\mathbb{C} \setminus (2i\pi \mathbb{Z})$. Then, $S_n$ is defined on $\mathbb{C}^2 \setminus (2i\pi \mathbb{Z})^2$, which implies that $m_n$ is defined on $\mathbb{C}^2$.

\ \vspace{-1ex} \\
The next lemma recalls some useful known properties of unsigned Stirling numbers. 

\begin{lem}
\label{LEM-Stirling}
We have for all $k \leq n$,
\begin{equation*}
{n \brack  n-k} \left(\frac{(n-k)!}{n!}\right)^2 \leq \frac{1}{k !}.
\end{equation*}
Moreover,
\begin{equation*}
    \lim_{n \rightarrow +\infty} {n \brack  n-k} \left(\frac{(n-k)!}{n!}\right)^2 = \frac{1}{2^k k!}.
\end{equation*}
\end{lem}
\begin{proof}
For the first property, we use Definition 26.8.3 in \cite{olver2010} of the unsigned Stirling numbers of first kind. The inequality follows by bounding the factors inside the sum by the maximum element $(n-1)!/(n-k-1)!$ and by noticing that there are $\binom{n-1}{k}$ elements in the sum.  The second property follows from formula 26.8.43 in \cite{olver2010} and the Stirling approximation. 
\end{proof}

\ \vspace{-1ex} \\
The last lemma gives an upper-bound for the functions $r_n$.

\begin{lem}
\label{LEM-bound-r}
For all $n \geq 1$ and for all $t \in \mathbb{R}^+$, there exists a constants $C_{t}$ such that
for any $v \in [-\pi,\pi]$, 
\begin{equation*}
    |r_n(t+i v)| \leq C_{t} \left|\frac{t+iv}{t+i(2\pi-|v|)}\right|^{n+1} \leq C_{t}.
\end{equation*}
\end{lem}
\begin{proof}
We have for any $v \in [-\pi,\pi]$, 
$$|r_n(t+i v)| \left|\frac{t+i(2\pi-|v|)}{t+iv}\right|^{n+1} \leq \sum_{\ell \in \mathbb{Z}^*} \left|\frac{t^2+(2\pi-|v|)^2}{t^2+(v+2\pi \ell)^2}\right|^{\frac{n+1}{2}}.$$
Hereafter, assuming that $v \in [-\pi,0]$, we have
\begin{align*}
    \sum_{\ell \in \mathbb{Z}^*} \left|\frac{t^2+(2\pi-|v|)^2}{t^2+(v+2\pi \ell)^2}\right|^{\frac{n+1}{2}}\!\!\! &= 1 + \sum_{\ell \geq 2} \left|\frac{t^2+(2\pi-|v|)^2}{t^2+(v+2\pi \ell)^2}\right|^{\frac{n+1}{2}}+ \sum_{\ell \leq -1} \left|\frac{t^2+(2\pi-|v|)^2}{t^2+(v+2\pi \ell)^2}\right|^{\frac{n+1}{2}}\\
    &\leq 1 + \sum_{\ell \geq 2} \left|\frac{t^2+\pi^2}{t^2+4\pi^2 (\ell-1)^2}\right|^{\frac{n+1}{2}}+ \sum_{\ell \leq -1} \left|\frac{t^2+\pi^2}{t^2+4\pi^2 \ell^2}\right|^{\frac{n+1}{2}}\\
    &= 1+ 2 \sum_{\ell \in \mathbb{N}^{*}} \left|\frac{t^2+\pi^2}{t^2+4\pi^2 \ell^2}\right|^{\frac{n+1}{2}} <+\infty,
\end{align*}
since we have assumed that $n+1 \geq 2$. By symmetry the previous upper-bound stays true for $v \in [0,\pi]$, which gives us the announced result.
\end{proof}

\subsection{Proof of Theorem \ref{T-LDP}}

For all $t,s \in \dR$, denote by $L_n$ the normalized cumulant generating function 
\begin{equation*}
    L_n(t,s)  = \frac{1}{n}\log m_n(t,s).
\end{equation*}
It follows from Lemma \ref{LEM-mn} that for all $t,s \in \dR$,
\begin{equation}
\label{LIMLN}
\lim_{n\rightarrow \infty} L_n(t,s) =L(t,s)=\log \left( \Bigl(\frac{\exp(t)-1}{t}\Bigr)\Bigl(\frac{\exp(s)-1}{s}\Bigr)\right)=
L(t)+L(s)
\end{equation}
where 
$$
L(t)=\log \left(\frac{\exp(t)-1}{t}\right).
$$
The function $L$ is finite and differentiable on all the real line.  Then, we deduce from 
the G\"artner-Ellis theorem, see e.g. Theorem 2.3.6 in \cite{Dembo1998}, that
the couple $(D_n/n,D'_n/n)$ satisfies an LDP with good rate function $I$ given by 
\eqref{DEFCI}.
\demend

\subsection{Proof of Theorem \ref{T-SLDP}}
Our goal is to establish, for all $x$ and $y$ in $]1/2,1[$, a sharp asymptotic expansion of the probability
\begin{equation}
\label{PPROBA}
\dP \Big(\frac{D_n}{n-1} \geq x, \frac{D'_n}{n-1} \geq y  \Big) = \sum_{k=\lceil (n-1)x \rceil}^{n-1}
\sum_{k'=\lceil (n-1)y \rceil }^{n-1} \dP(D_n=k, D'_n=k').
\end{equation}
For all $t,v,s,w \in \mathbb{R}$, we have
$$m_n(t+iv,s+iw) = \sum_{k = 0}^{n-1} \sum_{k'=0}^{n-1} e^{(t+iv)k+(s+iw)k'}\mathbb{P}(D_n=k,D'_n=k').$$
Therefore, for all $0\leq k,k'\leq n-1$ and in fact for all $k,k' \in  \dZ$, we obtain that
$$\mathbb{P}(D_n=k,D'_n=k') = e^{-t k -s k'}\frac{1}{(2\pi)^2} \int_{[-\pi, \pi]^2} \!\!\!\! m_n(t+iv,s+i w) e^{-ikv -ik' w}dv dw.$$
As $D_n$ and $D'_n$ are smaller than $n$, we clearly have $|m_n(t+iv,s+iw)|\leq e^{|t|n+|s|n}$.
Hence, it follows from \eqref{PPROBA} together with Fubini's theorem that for all $t,s>0$,
\begin{equation}
\label{PROBAGEQXGEQY}
\dP\Big(\frac{D_n}{n-1} \geq x, \frac{D'_n}{n-1} \geq y \Big) = 
\frac{1}{(2\pi)^2} \int_{[-\pi, \pi]^2} \!\!\!\!\! \!\!m_n(t+iv,s+iw)  
\Sigma_n(t+iv,s+iw) dv dw
\end{equation}
where
\begin{eqnarray*}
\Sigma_n(t+iv,s+iw) &=&\sum_{k = \lceil (n-1)x \rceil }^{+\infty} \sum_{k' = \lceil (n-1)y \rceil }^{+\infty}   e^{-k(t+i v)-k'(s+iw)} \\
&=& \Big(\frac{e^{\! - \lceil (n-1)x \rceil(t+iv)} }{1-e^{-(t+i v)}}\Big)
\Big(\frac{e^{ \! - \lceil (n-1)y \rceil(s+iw)} }{1-e^{-(s+iw)}}\Big).
\end{eqnarray*}
In the following, we choose $t=t_x$ and $s=t_y$. One can observe
that $t_x$ and $t_y$ are both positive since $x,y> 1/2$. Consequently, we deduce from
Lemma \ref{LEM-mn} and \eqref{PROBAGEQXGEQY}, together with the fact that $I(x)=xt_x-L(t_x)$
and $\{x_n\}=nx-\lceil (n-1)x \rceil$, that
\begin{equation}
\label{EQ-PROB-In}
   \dP\Big(\frac{D_n}{n-1} \geq x , \frac{D'_n}{n-1} \geq y \Big) =\Big(\frac{e^{-nI(x)-t_x \{ x_n\}}}{2\pi} \Big)\Big(\frac{e^{-nI(y)-t_y \{  y_n\}}}{2\pi} \Big)I_n
\end{equation}
where 
\begin{equation*}
    I_n = \int_{[-\pi, \pi]^2} \!\! e^{-n \varphi_x(v)-n \varphi_y(w)} f_n^1(v,w) dv dw,
\end{equation*}
with
\begin{equation*}
    \varphi_x(v) = -(L(t_x+iv)-L(t_x)- i x v)
\end{equation*}
and, for $a \in \{0,1\}$,
\begin{equation}
\label{DEFfna}
    f_n^{a}(v,w) = \Big(\frac{e^{i \{ x_n\} v}}{t_x+i v} \Big)\Big(\frac{e^{i \{ y_n\} w}}{t_y+i w} \Big)S_n^{a}(t_x+i v,t_y+i w).
\end{equation}
By using Lemmas \ref{LEM-Stirling} and \ref{LEM-bound-r}, we easily obtain that 
\begin{equation}
    \label{EQ-bound-fn}
    \big|f_n^{a}(v,w)\big| \leq \frac{e^{t_x t_y}}{t_x t_y} (1+a C_{t_x})(1+a C_{t_y}) .
\end{equation}
Hereafter, we split the integral $I_n$ into two parts,  $I_n=J_n+K_n$ where
\begin{eqnarray*}
J_n&=&\int_{[-\pi/2,\pi/2]^2} \!\!e^{-n \varphi_x(v)-n \varphi_y(w)} f_n^1(v,w) dv dw, \\
K_n&=&\int_{[-\pi,\pi]^2 \setminus[-\pi/2,\pi/2]^2}  \!\!e^{-n \varphi_x(v)-n \varphi_y(w)} f_n^1(v,w) dv dw.
\end{eqnarray*}
It follows from the last part of Lemma 5.1 in \cite{Bercu2024} that for all  $v \in \dR$ such that
$|v| \leq \pi$,
\begin{equation}
    \label{INEQ-L}
    -\Re(\varphi_x(v)) =  \Re (L(t_x+iv)-L(t_x)) \leq -\lambda_x\frac{|v|^2}{2}
\end{equation}
where
$$
\lambda_x=\frac{t_x^2 \sigma_x^2}{t_x^2+\pi^2}.
$$
Then, by symmetry, we obtain from \eqref{EQ-bound-fn} and \eqref{INEQ-L} that
\begin{eqnarray}
    \nonumber
    |K_n| &\leq & \frac{e^{t_x t_y}}{t_x t_y} (1+C_{t_x})(1+C_{t_y}) \int_{[-\pi,\pi]^2 \setminus[-\pi/2,\pi/2]^2} e^{-n\Re(\varphi_x(v)+ \varphi_y(w))} dv dw,\\
    &\leq& \frac{4 e^{t_x t_y}}{ t_x t_y} (1+C_{t_x})(1+C_{t_y}) \Big(K_n^1+K_n^2+K_n^3\Big)
    \label{INEQ-Kn}
\end{eqnarray}
where, via standard Gaussian calculations, $K_n^1, K_n^2$ and $K_n^3$ are such that
\begin{align*}
K_n^1 &= \int_{\pi/2}^\pi \!\!\exp\Bigl(-\frac{n \lambda_x v^2}{2}\Bigr) dv
\!\int_{\pi/2}^\pi \!\!\exp\Bigl(-\frac{n \lambda_y w^2}{2}\Bigr) dw \leq \frac{4}{n^2 \pi^2 \lambda_x \lambda_y} \exp\Bigl(-\frac{n \pi^2}{8}\big(\lambda_x+\lambda_y\big)\!\Bigr), \\
K_n^2 &= \int_{0}^{\pi/2} \!\!\exp\Bigl(-\frac{n \lambda_x v^2}{2}\Bigr) dv
\!\int_{\pi/2}^\pi \!\!\exp\Bigl(-\frac{n \lambda_y w^2}{2}\Bigr) dw\leq \frac{\sqrt{2}}{n \lambda_y \sqrt{n\pi \lambda_x}} \exp\Bigl(-\frac{n \pi^2}{8}\lambda_y\!\Bigr), \\
K_n^3&= \int_{\pi/2}^{\pi}\!\! \exp\Bigl(-\frac{n \lambda_x v^2}{2}\Bigr) dv
\!\int_{0}^{\pi/2} \!\!\exp\Bigl(-\frac{n \lambda_y w^2}{2}\Bigr) dw\leq \frac{\sqrt{2}}{n \lambda_x \sqrt{n\pi \lambda_y}} \exp\Bigl(-\frac{n \pi^2}{8}\lambda_x\!\Bigr).
\end{align*}
Consequently, we clearly deduce from \eqref{INEQ-Kn} that $K_n$ goes exponentially fast to zero. It only remains to study the integral $J_n$ through the following three quantities,
\begin{align*}
    J_n =& \int_{[-\pi/2,\pi/2]^2} e^{-n \varphi_x(v)-n \varphi_y(w)} (f_n^1(v,w)-f_n^0(v,w)) dv dw\\
    &+ \int_{[-\pi/2,\pi/2]^2} e^{-n \varphi_x(v)-n \varphi_y(w)} (f_n^0(v,w)-f_n^0(0,0)) dv dw\\
    & + f_n^0(0,0) \int_{[-\pi/2,\pi/2]^2} e^{-n \varphi_x(v)-n \varphi_y(w)} dv dw = J^1_n+J^2_n+J^3_n.
\end{align*}
On the one hand, we have from \eqref{DEFfna} that 
$$
f_n^0(0,0) = \frac{1}{t_x t_y} \sum_{k=0}^{n-1} a_{n,k}
\hspace{1cm}\text{where} \hspace{1cm}
a_{n,k} = {n \brack  n-k} \left(\frac{(n-k)!}{n!}\right)^2(t_x t_y)^k.
$$
We obtain from Lemma \ref{LEM-Stirling} that 
$$ \lim_{n \rightarrow + \infty} a_{n,k} = \frac{t_x^k t_y^k}{2^k k!} 
\hspace{1cm}\text{and} \hspace{1cm}
0 \leq a_{n,k} \leq \frac{t_x^k t_y^k}{k !}.$$
Therefore, we can apply the dominated convergence theorem to get that
\begin{equation*}
    \lim_{n \rightarrow + \infty} f_n^0(0,0) = \frac{e^{\frac{t_x t_y}{2}}}{t_x t_y}.
\end{equation*}
On the other hand, recalling \eqref{INEQ-L} and using a slight extension of the usual Laplace method given by Lemma 5.2 in \cite{Bercu2024}, we obtain that
$$
\int_{[-\pi/2,\pi/2]^2} e^{-n \varphi_x(v)-n \varphi_y(w)} dv dw  = \frac{2\pi}{n \sigma_x \sigma_y} (1+o(1)),
$$
which ensures that
\begin{equation}
    \label{INEQ-Jn3}
    J_n^3 = \frac{2 \pi}{n \sigma_x \sigma_y} \frac{e^{\frac{t_x t_y}{2}}}{t_x t_y}(1+o(1)).
\end{equation}
Furthermore, one can observe that $S^0_n(t_x+iv,t_y+iw)$ is a bounded Lipschitz function on $[-\pi/2,\pi/2]^2$, uniformly in $n$. Indeed, by using Lemma \ref{LEM-Stirling}, we have 
\begin{equation*}
    \Big|S^0_n(t_x+iv,t_y+iw)\Big| \leq \sum_{k=0}^{n-1} \frac{1}{k!}
    \bigl(t_x + |v|\bigr)^k \bigl(t_y + |w|\bigr)^k\leq \exp \Bigl( \Bigl(t_x + \frac{\pi}{2}\Bigr) \Bigl(t_y + \frac{\pi}{2}\Bigr)\Bigr),
\end{equation*}    
and  
\begin{align*}  
    \Big| \frac{\partial S^0_n(t_x+iv,t_y+iw)}{\partial v} \Bigr| &\leq \sum_{k=1}^{n-1} 
    \frac{k(t_x + |v|\bigr)^{k-1} \bigl(t_y + |w|\bigr)^k}{k!}, \\
    & \leq \Bigl(t_y + \frac{\pi}{2}\Bigr)\exp \Bigl( \Bigl(t_x + \frac{\pi}{2}\Bigr) \Bigl(t_y + \frac{\pi}{2}\Bigr)\Bigr).
\end{align*}    
Since $S^0_n(t_x+iv,t_y+iw)=S^0_n(t_y+iw,t_x+iv)$, 
it implies that there exists a positive constant $C$ which does not depend on $n$, such that
$$
|f_n^0(v,w) - f_n^0(0,0)| 
\leq C(|v|+|w|).
$$
Applying this bound, inequality \eqref{INEQ-L} and the usual Laplace method, we obtain 
\begin{equation}
\label{INEQ-Jn2}
    |J_n^2| \leq  C \int_{[-\pi/2,\pi/2]^2} e^{-n \Re (\varphi_x(v)+ \varphi_y(w))}(|v|+|w|) dv dw.
\end{equation}
Hence, as for $K_n$, $J_n^2$ goes exponentially fast to zero. Hereafter, we just have to study the last term $J_n^1$ in order to conclude. By using Lemma \ref{LEM-bound-r}, we have, for all $v \in [-\pi/2,\pi/2]$ and for all $n \geq 1$,
\begin{equation*}
|r_n(t_x+i v)| \leq  C_{t_x} (q_x)^{n+1} \hspace{1cm}\text{where} \hspace{1cm}
 q_x = \left( \frac{t_x^2+(\pi/2)^2}{t_x^2+(3\pi/2)^2} \right)^{1/2}.
\end{equation*}
One can observe that $0<q_x<1$, which means that $r_n(t_x+i v)$ goes exponentially fast to zero.
Then, by using once again Lemma \ref{LEM-Stirling}, we obtain from \eqref{DEFfna} that for all $(v,w) \in [-\pi/2,\pi/2]^2$,
 \begin{align*}
  &   \big|f_n^1(v,w) - f_n^0(v,w)\big|
     \leq   \frac{1}{t_xt_y}\sum_{k=0}^{n-1} \frac{(t_x+|v|)^k (t_y + |w|)^k }{k!} 
     \big| R_{n-k}(t_x+v,t_y+iw) -1 \big|, \\ 
 & \leq  \frac{C_{t_x} (q_x)^{n+1}}{t_xt_y} \exp \Bigl( \frac{1}{q_x}\Bigl(t_x + \frac{\pi}{2}\Bigr) \Bigl(t_y + \frac{\pi}{2}\Bigr)\Bigr) + \frac{C_{t_y} (q_y)^{n+1}}{t_xt_y} \exp \Bigl( \frac{1}{q_y}\Bigl(t_x + \frac{\pi}{2}\Bigr) \Bigl(t_y + \frac{\pi}{2}\Bigr)\Bigr) \\
& \hspace{4cm} + 
  \frac{C_{t_x} C_{t_y}(q_x q_y)^{n+1}}{t_xt_y} \exp \Bigl( \frac{1}{q_xq_y}\Bigl(t_x + \frac{\pi}{2}\Bigr) \Bigl(t_y + \frac{\pi}{2}\Bigr)\Bigr). 
 \end{align*}
Consequently, we deduce from inequality \eqref{INEQ-L} that
$J_n^1$ goes exponentially fast to zero.
Finally, the only contribution for the integral $I_n$ is given by \eqref{INEQ-Jn3} 
and \eqref{EQ-PROB-In} clearly leads to \eqref{SLDPR}. 
From now on, we shall carry out the proof of \eqref{SLDPR2}. Equation \eqref{PROBAGEQXGEQY} becomes, for all $t,s<0$,
\begin{equation*}
   \dP\Big(\frac{D_n}{n-1} \leq 1-x, \frac{D'_n}{n-1} \leq 1-y \Big) 
   =\frac{1}{(2\pi)^2} \int_{[-\pi, \pi]^2} \!\!\!\!\!\!\!\!\!m_n(t+iv,s+iw)  
\Sigma_n(t+iv,s+iw) dv dw
\end{equation*}
where
\begin{eqnarray*}
\Sigma_n(t+iv,s+iw) &=&\sum_{k =-\infty}^{\lfloor (n-1)(1-x) \rfloor } 
\sum_{k' = -\infty}^{\lfloor (n-1)(1-y) \rfloor }   e^{-k(t+i v)-k'(s+iw)} \\
&=& \Big(\frac{e^{\! - (\lfloor (n-1)(1-x) \rfloor +1)(t+iv)} }{e^{-(t+i v)}-1}\Big)
\Big(\frac{e^{ \! - (\lfloor (n-1)(1-y) \rfloor +1)(s+iw)} }{e^{-(s+iw)}-1}\Big).
\end{eqnarray*}
In the following, we choose $t=t_{1-x}$ and $s=t_{1-y}$. In contrast with the previous case, we have $t_{1-x}=-t_x<0$, $t_{1-y}=-t_y<0$ and $I(1-x)=I(x)$, $I(1-y)=I(y)$. Moreover, a simple computation shows that for any $z \in \dR$, 
$$
n(1-z)-\lfloor (n-1)(1-z) \rfloor -1 = -(nz-\lceil(n-1)z \rceil).
$$
Then, via the same arguments as in the proof of \eqref{EQ-PROB-In}, we find that
\begin{align*}
\label{NEQ-PROB-In}
   \dP\Big(\frac{D_n}{n-1} \leq 1-x , \frac{D'_n}{n-1} \leq 1-y \Big)
   & = \Big(\frac{e^{-nI(1-x)-t_{1-x} \{ x_n\}}}{2\pi} \Big)\Big(\frac{e^{-nI(1-y)-t_{1-y} \{  y_n\}}}{2\pi} \Big)I_n, \\
   & =\Big(\frac{e^{-nI(x)+t_x \{ x_n\}}}{2\pi} \Big)\Big(\frac{e^{-nI(y)+t_y \{  y_n\}}}{2\pi} \Big)I_n
\end{align*}
where 
\begin{equation*}
    I_n = \int_{[-\pi, \pi]^2} \!\! e^{-n \varphi_{1-x}(v)-n \varphi_{1-y}(w)} f_n^1(v,w) dv dw,
\end{equation*}
with
\begin{equation*}
    \varphi_{1-x}(v) = -(L(-t_x+iv)-L(-t_x)- i (1-x) v)
\end{equation*}
and, for $a \in \{0,1\}$,
\begin{equation*}
\label{DEFfna}
    f_n^{a}(v,w) = \Big(\frac{e^{-i \{ x_n\} v}}{t_x-i v} \Big)\Big(\frac{e^{-i \{ y_n\} w}}{t_y-i w} \Big)S_n^{a}(-t_x+i v,-t_y+i w).
\end{equation*}
Therefore, we just have to make use of the same calculations as previously done to show \eqref{SLDPR2}. Finally, we establish \eqref{SLDPR3} and \eqref{SLDPR4} using the same lines, 
which completes the proof of Theorem \ref{T-SLDP}.
\demend

\section{Proof of the construction of the probability space}
\label{S-PPS}
Recall the construction made in Section \ref{S-PS}. We call fiber and we denote by $\mathcal{F}_n(\ell)$ the set of cells $\big\{\mathcal{C}_n(k,\ell) \big\slash k\in \{1,2,\dots,n+1\}\big\}$ for $\ell\in\{1,2,\dots,n+1\}$.

\begin{figure}[h!]
	\centering
	\includegraphics[scale=0.6]{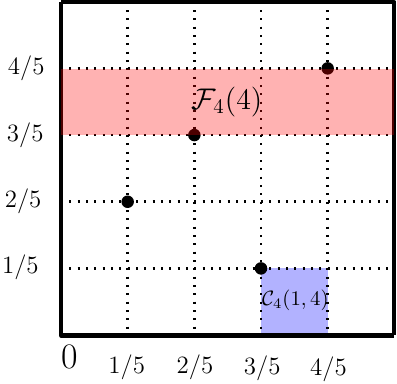}
	\caption{Graphical representation of $\pi_4 = (2314)$ to the left. We colored in light blue the cell $\mathcal{C}_4(1,4)$ and in light red the cells belonging to the fiber $\mathcal{F}_4(4)$.}
	\label{fig:graph_rep}
\end{figure}

\noindent
The following lemma states that the number of cells in each fiber with $\Delta D_{n+1} = +1$ depends only on the number of descent of $\pi_n$ and has a closed formula.

\begin{lem}\label{lem:count}
	For each $\ell\in \{1,2,\dots,n+1\}$, the fiber $\mathcal{F}_n(\ell)$ in the graphical representation has $n-D_n$ cells with increment $+1$ and $D_n+1$ with increment $0$. In other words, $\text{Card} \big\{k \big\slash \Delta D_{n+1}(k,\ell)=1 \big\} = n-D_n$, $\text{Card} \big\{k \big\slash \Delta D_{n+1}(k,\ell)=0 \big\} = D_n+1$.
\end{lem}
\begin{proof}
For each point $p=(k/(n+1),\ell/(n+1))$ in $GR(\pi_n)$, we color in red the cells that are adjacent to the left of the line $k/(n+1)\times [0,1]$ and that are bigger than $p$ and we color in blue the cells adjacent to the right of the line $k/(n+1)\times [0,1]$ and that are smaller than the point $p$.
	\begin{figure}[h!]
		\centering
		\includegraphics[scale=0.6]{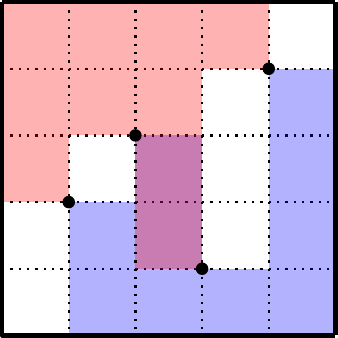}
		\caption{Coloring of the graphical representation associated to $\pi_4 = (2314)$.}
		\label{fig:graph_rep_lines}
	\end{figure}
	
\noindent
Let $R_n(k,\ell) = 1$ if the cell $\mathcal{C}_n(k,\ell)$ is colored red and $0$ otherwise, and
let $B_n(k,\ell)=1$ if the cell $\mathcal{C}_n(k,\ell)$ is colored blue and $0$ otherwise.
With our representation, note that every point induce exactly one cell colored in each fiber, that is
for all $(k, \ell)$ in $\{1,\dots,n+1\}^2$,
\begin{equation}
\label{eq:prop1}
		R_n(k,\ell)+B_n(k+1,\ell) = 1.
\end{equation}
\noindent
This property can be easily checked in  Figure $4$ A).
Let $\ell\in\{1,2,\dots,n+1\}$ be the fiber under study. A deeper analysis shows that there are four cases for each cell.
\begin{itemize}
		\item[L)] The cell belongs to the left boundary, i.e. $\mathcal{C}_n(1,\ell)$ . In this case, it cannot be colored blue and it has an increment $+1$ if the cell is colored red, otherwise the increment is 0, see Figure $4$ B). 
The function $R_n(1,\ell)$ is equal to the indicator function of the cell having increment +1.
		\item[R)] The cell belongs to the right boundary, i.e. $\mathcal{C}_n(n+1,\ell)$. In ths case, 
it cannot be colored red and it has an increment $+1$ if the cell is colored blue, otherwise the increment is 0, see Figure  $4$ B). 
The function $B_n(n+1,\ell)$ is equal to the indicator function of the cell having increment +1.
		\item[I)]The cell belongs to the inner columns, i.e. $\mathcal{C}_n(k,\ell)$ for $k\in \{2,3,\dots,n\}$. In this case, it depends on whether or not there is a descent at $k-1$ for $\pi_n$.
			\item[I0)] If $\pi_n(k-1)<\pi_n(k)$, then the cell adds an increment $+1$ if it is colored only red or only blue, and in this case it cannot be simultaneously colored blue and red, 
			see Figure $4$ B). 
			 The indicator function of the cell having increment +1 can be written as 
			 $R_n(k,\ell)+B_n(k,\ell)$.
			\item[I1)] If $\pi_n(k-1)>\pi_n(k)$, then the cell adds an increment $+1$ if it is colored red and blue, otherwise the increment is 0. Notice that when $D_n(k-1)=1$, all the cells in  $\{\mathcal{C}_n(k,\ell) \big \slash\ell \in \{1,2,\dots,n+1\}\}$ are colored by at least one color,
			see Figure $4$ B). 
			 It can easily be checked that the indicator function of the cell having increment +1 can be written as $R_n(k,\ell)+B_n(k,\ell)-1$.
\begin{figure}[h!]
				\begin{subfigure}{.35\textwidth}
				\centering
				\includegraphics[scale=0.6]{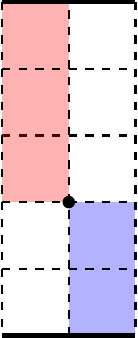}
				\caption{}
				\label{fig:yellow_or_blue}
			\end{subfigure}
			\begin{subfigure}{.55\textwidth}
				\centering
				\includegraphics[scale=0.6]{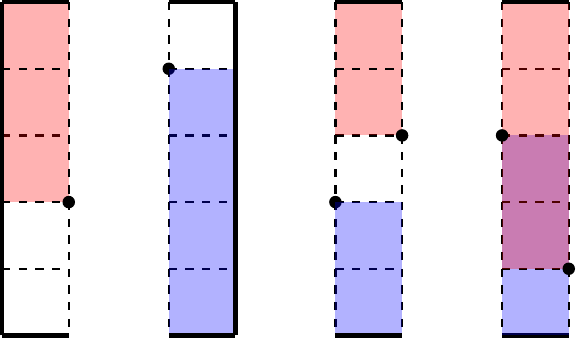}
				\caption{}
			\label{fig:cases}
		\end{subfigure}
			\caption{Figure A): In each fiber, there is exactly one red or blue cell induced by each point of the permutation, as stated by \eqref{eq:prop1}. Figure B): From left to right, we show the cases in R), L), I0) and I1), respectively. }
		\end{figure}
	\end{itemize}
For each $\ell\in \{1,2,\dots,n+1\}$ and for the fiber $\cF_n(\ell)$, we obtain that the sum of cells with increment +1 is given by $\Sigma_n^1+\Sigma_n^2+\Sigma_n^3$ where
$\Sigma_n^1 =R_n(1,\ell) + B_n(n+1,\ell)$,
\begin{eqnarray*}
		\Sigma_n^2 &=&\sum_{\stackrel{k \in\{2,\dots,n\}}{\pi_n(k-1)<\pi_n(k)}} (R_n(k,\ell)+B_n(k,\ell)), \\
		\Sigma_n^3 &=&\sum_{\stackrel{k \in\{2,\dots,n\}}{\pi_n(k-1)>\pi_n(k)}}  (R_n(k,\ell)+B_n(k,\ell)-1).
\end{eqnarray*}
Finally, we deduce from \eqref{DEFDN} and \eqref{eq:prop1} that 
$$
\Sigma_n^1+\Sigma_n^2+\Sigma_n^3 = \sum_{k=1}^n (R_n(k,\ell)+B_n(k+1,\ell))-D_n=n-D_n.
$$
Moreover, as the cells can have either increments +1 or 0, we find that the sum of cells with increment 0 is given by $n+1-n+D_n = D_n+1$.
\end{proof}
\begin{cor}\label{cor:1}
For every permutation $\pi_{n} \in \cS_{n}$, out of the $(n+1)^2$ cells in our graphical representation, there are exactly $(n+1)(n-D_n)$ cells with increment $+1$.
\end{cor}

\noindent
We are now in the position to prove Theorem \ref{T-PROBASPACE}.
\begin{proof}[Proof of Theorem \ref{T-PROBASPACE}]
We claim that it is enough to show that 
\begin{equation}
\label{N11}
N_{n+1} (1,1,\pi_n)= (n-D_n)(n-D^\prime_n)+n.
\end{equation} 
As a matter of fact,  for a permutation $\pi_n \in \cS_{n}$, denote $\bar{\pi}_n = n-\pi_n$. The permutation $\bar{\pi}_n$ is a bijective involution, which exchanges the descents (resp. ascents) of $\pi_n$ into ascents (resp. descents) of $\bar{\pi}_n$ and it does the same for $\pi_n^{-1}$ and $\bar{\pi}_n^{-1}$. Since in total the ascents and descents sum up to $n-1$, we have that the descents for $\bar{\pi}_{n}$ are equal to $D_n+1$ and the descents for $\bar{\pi}_n^{-1}$ are equal to $D'_n+1$. Therefore, we have
$$
N_{n+1} (0,0,\pi_n) = N_{n+1} (1,1,\bar{\pi}_n) = (D_n+1)(D'_n+1)+n.
$$
Moreover, for $N_{n+1} (1,0,\pi_n)$, it follows from to Corollary \ref{cor:1} that
$$
N_{n+1} (1,0,\pi_n)+ N_{n+1} (1,1,\pi_n) = (n+1)(n-D_n),
$$
which clearly leads via \eqref{N11} to 
\begin{eqnarray*}
N_{n+1} (1,0,\pi_n) &=& (n+1)(n-D_n) - (n-D_n)(n-D^\prime_n)-n, \\
&=&(n-D_n)(D'_n+1)-n.
\end{eqnarray*}
Consequently, as
$$
N_{n+1} (1,1,\pi_n)+N_{n+1} (1,0,\pi_n)+N_{n+1} (0,1,\pi_n)+N_{n+1} (0,0,\pi_n)=(n+1)^2,
$$
we find that
$$
N_{n+1} (0,1,\pi_n)=(D_n+1)(n-D'_n)-n.
$$
Hereafter, we focus our attention on the proof of \eqref{N11} 
by induction on the size of the permutation.
\ \vspace{1ex}\\
The base case $n=1$ is easy to check since there is only one possible permutation, the identity for 1 element, which is equal to its inverse. Consequently, we have
$N_{2} (1,1,\pi_1) = 2 = (1-D_1)(1-D'_1)+1$ where $D_1=D_1'=0$.
\ \vspace{1ex}\\	
For the induction step, suppose that \eqref{N11} holds true for $n\geq 1$ and 
let $\pi_{n+1}$ be a permutation of size $\cS_{n+1}$. We define $\pi_n$ as the permutation induced by $\pi_{n+1}$ when one takes out $\pi_{n+1}(n+1)$, see Figure \ref{supress} below. 
More precisely, we see $\pi_n$ as the permutation $(\pi_{n+1}(k) \big \slash k \in \{1,\dots,n\})$ with respect to the relative order of the images present in this vector from left to right.
	\begin{figure}[h!]
		\centering
		\includegraphics[scale=0.6]{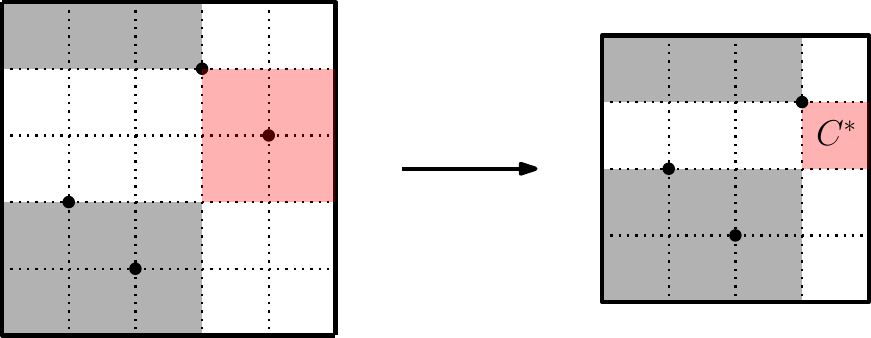}
		\caption{Permutation $\pi_3 = (213)$ (to the right) induced from $\pi_{4} = (2143)$ (to the left) when taking out $\pi_{4}(4) = 3$. We keep track of the increment behavior of the red cell during the induction.}
		\label{supress}
	\end{figure}

\noindent
We denote by $C^*$ the cell in $\pi_n$ where the point $u$ belongs to create the whole permutation $\pi_{n+1}$. It is important to note that if the cell $C^* = \mathcal{C}_n(n+1,\ell)$ where
$\ell \in  \{1,\dots,n+1\}$, then the increment behavior of this cell fixes the relative order of $\pi_n(n)$ and the values of the inverse for $\ell$ and $\ell-1$, that is $\pi_n^{-1}(\ell)$ and $\pi_n^{-1}(\ell-1)$, this are the three points we plot in each case of what follows. It is also important to observe that the increment behavior of the gray cells in Figure \ref{supress} remain invariant by this transformation. There are four possible cases depending on $u \in C^*$.
	\begin{itemize}
		\item[1)] \underline{$\Delta D_{n+1}=1$ and $\Delta D'_{n+1}=1$}. In this case, the relative order of $\pi_n(n)$, $\pi_n^{-1}(\ell)$ and $\pi_n^{-1}(\ell-1)$ is given in Figure \ref{fig:1case}.
			\begin{figure}[h!]
			\centering
			\includegraphics[scale=0.5]{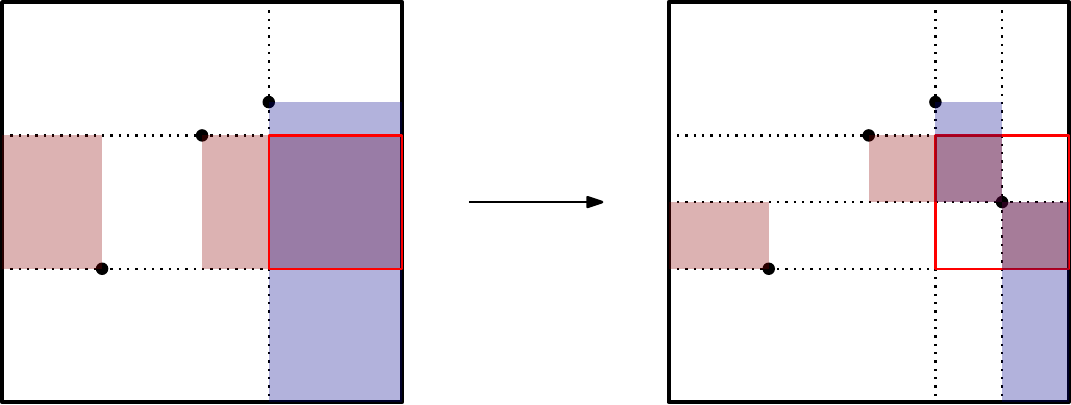}
			\caption{Relative order of $\pi_n(n)$, $\pi_n^{-1}(\ell)$ and $\pi_n^{-1}(\ell-1)$ in case 1). The blue cells are cells with $\Delta D_{n+1}=1$ and the red cells are cells with $\Delta D'_{n+1}=1$.}
			\label{fig:1case}
		\end{figure}

\noindent
In this case, as $D_{n+1}=D_n+1$ and $D'_{n+1}=D'_n+1$, we have from \eqref{N11} that 
		\begin{align*}
			N_{n+2} (1,1,\pi_{n+1})
			&= N_{n+1} (1,1,\pi_n) -1+2, \\
			&=(n-D_n)(n-D'_{n})+n+1\\
			&=(n+1-D_{n+1})(n+1-D'_{n+1})+n+1,
		\end{align*} 
		where the first equality follows since the only cell which is simultaneously red and blue becomes two cells that are simultaneously red and blue.
		\item[2)] \underline{$\Delta D_{n+1}=1$ and $\Delta D'_{n+1}=0$}.
		In this case, the relative order of $\pi_n(n)$, $\pi_n^{-1}(\ell)$ and $\pi_n^{-1}(\ell-1)$ is given in Figure \ref{fig:3case}. 
		\begin{figure}[h!]
			\centering
			\includegraphics[scale=0.5]{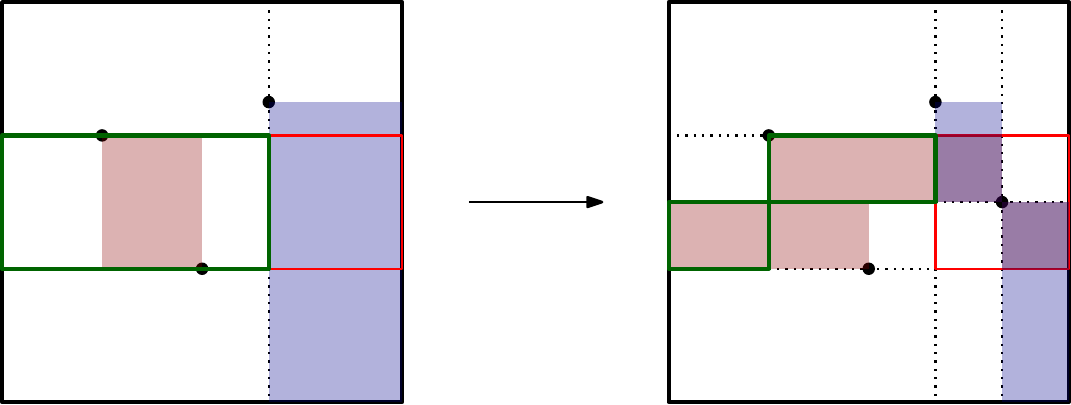}
			\caption{Relative order of $\pi_n(n)$, $\pi_n^{-1}(\ell)$ and $\pi_n^{-1}(\ell-1)$ in case 2). The blue cells are cells with $\Delta D_{n+1}=1$ and the red cells are cells with $\Delta D'_{n+1}=1$.}
			\label{fig:3case}
		\end{figure}

\noindent	
		What remains is to count the number of cells in the delimited green box in the right of Figure \ref{fig:3case} that are also colored in blue. 
We do this bijectively from the left. The quantity of cells in the delimited green box to the right having $\Delta D_{n+2}=1$ are in bijection with the ones having $\Delta D_{n+1} = 1$ in the green box to the left. One has from Lemma \ref{lem:count} applied to $\pi_{n}$ that there are $n-D_n-1$ cells in the delimited green part, the $-1$ coming from the cell $C^*$, which is colored blue, meaning that $\Delta D_{n+1}=1$. Consequently, as $D_n = D_{n+1}+1$ and $D'_n = D'_{n+1}$,
we have from \eqref{N11} that
\begin{align*}
	N_{n+2} (1,1,\pi_{n+1})
	&= N_{n+1} (1,1,\pi_{n}) +n-D_n-1+2\\
	&=(n-D_n)(n-D'_{n})+n+1+(n-D_n)\\
	&=(n-D_{n})(n+1-D'_{n})+n+1\\
	&=(n+1-D_{n+1})(n+1-D'_{n+1})+n+1.
\end{align*}
			
		\item[3)] \underline{$\Delta D_{n+1}=0$ and $\Delta D'_{n+1}=1$}.
		It follows from a similar argument as in case 2) using Lemma \ref{lem:count} applied to 
		$\pi_n^{-1}$.
		
		\item[4)] \underline{$\Delta D_{n+1}=0$ and $\Delta D'_{n+1}=0$}.
	    It follows by joining the argument in cases 2) and 3) using once again Lemma \ref{lem:count}
	    applied to $\pi_n$ and $\pi_n^{-1}$ simultaneously.
	\end{itemize}

\end{proof}

\section{Proof of the martingale results}
\label{S-PMG}
Finally, we shall proceed with the proof of the martingale results concerning 
the couple$(D_n, D^\prime_n)$ such as the functional central limit theorem, 
the quadratic strong law and the law of iterated logarithm.

\begin{proof}[Proof of Proposition \ref{P-FCLT}]
We already saw from \eqref{IPMN} that the predictable quadratic variation
$\langle M \rangle_n$ of the locally square integrable martingale
$(M_n)$ is given by 
\begin{equation*}
\langle M \rangle_n  =\sum_{k=1}^{n-1} \begin{pmatrix}
(k-D_k)(D_k+1) & k\\
k & (k-D_k^\prime)(D_k^\prime+1)
\end{pmatrix}.
\end{equation*}
It follows from the almost sure convergences
\begin{equation}
\label{ASCVGDNDPN}
\lim_{n \rightarrow \infty}\frac{D_n}{n} = \frac{1}{2} \hspace{1cm} \text{and} \hspace{1cm}
\lim_{n \rightarrow \infty}\frac{D_n^\prime}{n} = \frac{1}{2}
\end{equation}
together with the classical Toeplitz lemma that
\begin{equation}
\label{ASCVGPQVN}
\lim_{n \rightarrow \infty}\frac{1}{n^3} \langle M \rangle_n= \frac{1}{12} 
\begin{pmatrix}
1 & 0\\
0 & 1
\end{pmatrix}
\hspace{1cm} \text{a.s.}
\end{equation} 
Consequently, we immediately obtain from \eqref{ASCVGPQVN} that for all $t \geq 0$,
\begin{equation}
\label{ASCVGPQVNT}
\lim_{n \rightarrow \infty}\frac{1}{n^3} \langle M \rangle_{\lfloor nt \rfloor}= \frac{t^3}{12} 
\begin{pmatrix}
1 & 0\\
0 & 1
\end{pmatrix}
\hspace{1cm} \text{a.s.}
\end{equation} 
It is not hard to see that $(M_n)$ satisfies Lindeberg's condition since for all $n \geq 1$,
the increments $||\Delta M_n|| \leq 2n$.
Therefore, we deduce from the functional central limit theorem for martingales given e.g. in Theorem 2.5 of \cite{Durrett1978} that
\begin{equation}
\label{FCLTMART}
\Big(\frac{1}{\sqrt{n^{3}}} M_{\lfloor nt \rfloor}, t \geq 0\Big) \Longrightarrow \big( B_t, t \geq 0 \big)
\end{equation}
where $\big( B_t, t \geq 0 \big)$ is a continuous two-dimensional centered Gaussian process starting from the origin with covariance matrix given, for all $0<s \leq t$, by
$$
\dE[B_s B_t^T]= \frac{s^{3}}{12}
\begin{pmatrix}
1 & 0\\
0 & 1
\end{pmatrix}.
$$
Finally, \eqref{DEFMN} and \eqref{FCLTMART} lead to
\eqref{FCLTDN} where $W_t=B_t/t^2$,
which completes the proof of Proposition \ref{P-FCLT}.
\end{proof}

\begin{proof}[Proof of Proposition \ref{P-QSLLIL}]
We shall make use of the quadratic strong law for scalar martingales given e.g. by Theorem 3 in \cite{Bercu2004}. For any vector $u$ of $\dR^2$, denote $M_n(u)=\langle u, M_n \rangle$.
It follows from \eqref{DEFMN} that
\begin{equation}
\label{DEFMNU}
M_n(u)= n \left(\langle u,V_n\rangle 
- \frac{(n-1)}{2} \langle u,v\rangle  \right).
\end{equation}
We obtain from \eqref{DECVN} and \eqref{DEFMNU} that $M_n(u)$ can be rewritten in the additive form
\begin{equation}
\label{MNUS}
M_n(u)= \sum_{k=1}^{n} k \,\varepsilon_{k}(u)
\end{equation}
where $\varepsilon_{n+1}(u)=\langle u, \xi_{n+1}-\pi_n \rangle$ and $\pi_n$ is the vector of 
$\dR^2$ given by $\pi_n=(p_n,p^\prime_n)^T$. One can observe that $(M_n(u))$ is a
locally square integrable scalar martingale. Since for any vector $u$ of $\dR^2$,
\begin{eqnarray*}
\dE\bigl[\bigl(M_{n+1}(u)-M_{n}(u)\bigr)^2 | \cF_n\bigr] &=&
u^T \dE\bigl[(M_{n+1}-M_{n})(M_{n+1}-M_{n})^T | \cF_n\bigr]u, \\
&=&u^T ( \langle M\rangle_{n+1} - \langle M\rangle_{n}) u,
\end{eqnarray*}
we deduce from \eqref{IPMN} and \eqref{MNUS} that
\begin{equation*}
\label{CDM2ESPNU}
\dE\bigl[\varepsilon^2_{n+1}(u) | \cF_n\bigr] 
=\frac{1}{(n+1)^2}
u^T
\begin{pmatrix}
(n-D_n)(D_n+1) & n\\
n & (n-D_n^\prime)(D_n^\prime+1)
\end{pmatrix}
u.
\end{equation*}
Consequently, we obtain from \eqref{ASCVGDNDPN} that
\begin{equation}
\label{CDM2ESPNU}
\lim_{n \rightarrow \infty}
\dE\bigl[\varepsilon^2_{n+1}(u) | \cF_n\bigr] 
=\frac{1}{4}||u||^2 \hspace{1cm} \text{a.s.}
\end{equation}
Moreover, we already saw that for all $n \geq 1$, the increments $||\Delta M_n|| \leq 2n$ which implies that $|\varepsilon_{n}(u) | \leq 2 ||u||$, leading to
$$
\sup_{n \geq 0} \dE\bigl[\varepsilon^4_{n+1}(u) | \cF_n\bigr] < \infty \hspace{1cm} \text{a.s.}
$$
Therefore, it follows from the quadratic strong law for scalar martingales that
\begin{equation}
\label{PQSL1}
\lim_{n\rightarrow \infty} 
\frac{1}{\log s_n} \sum_{k=1}^n f_k\Big( \frac{M_k^2(u)}{s_k} \Big)  =\frac{1}{4}||u||^2
\hspace{1cm} \text{a.s.}
\end{equation}
where
$$
f_n=\frac{n^2}{s_n}
\hspace{1cm} \text{and}\hspace{1cm} s_n=\sum_{k=1}^n k^2.
$$
We clearly have that $n f_n$ converges to $3$
which immediately implies that $f_n$ goes to zero. Consequently, the quadratic strong law \eqref{PQSL1} reduces to
\begin{equation}
\label{PQSL2}
\lim_{n\rightarrow \infty} 
\frac{1}{\log n} \sum_{k=1}^n  \frac{M_k^2(u)}{k^4}   =\frac{1}{12}||u||^2
\hspace{1cm} \text{a.s.}
\end{equation}
Therefore, we deduce from \eqref{DEFMNU} and \eqref{PQSL2} that for any vector $u$ of $\dR^2$,
\begin{equation}
\label{PQSL3}
\lim_{n\rightarrow \infty} 
\frac{1}{\log n} \sum_{k=1}^n  \frac{1}{k^2}
u^T   \begin{pmatrix}
\frac{D_k}{k} - \frac{1}{2} \\
\frac{D_k^\prime}{k} - \frac{1}{2}
\end{pmatrix}
\begin{pmatrix}
\frac{D_k}{k} - \frac{1}{2} \\
\frac{D_k^\prime}{k} - \frac{1}{2}
\end{pmatrix}^{\!T}
\!\!u
=\frac{1}{12}||u||^2
\hspace{1cm} \text{a.s.}
\end{equation}
By virtue of the second part of Proposition 4.2.8 in \cite{Duflo1997}, we obtain from \eqref{PQSL3} that
\begin{equation}
\label{PQSL4}
\lim_{n\rightarrow \infty} 
\frac{1}{\log n} \sum_{k=1}^n  \frac{1}{k^2}
   \begin{pmatrix}
\frac{D_k}{k} - \frac{1}{2} \\
\frac{D_k^\prime}{k} - \frac{1}{2}
\end{pmatrix}
\begin{pmatrix}
\frac{D_k}{k} - \frac{1}{2} \\
\frac{D_k^\prime}{k} - \frac{1}{2}
\end{pmatrix}^{\!T}
=\frac{1}{12}\begin{pmatrix}
1 & 0\\
0 & 1
\end{pmatrix}
\hspace{1cm} \text{a.s.}
\end{equation}
Furthermore, we obviously obtain \eqref{QSLTRACEDN} 
by taking the trace on both sides of \eqref{PQSL4}.
It only remains to prove the law of iterated logarithm \eqref{LILTRACEDN}. We clearly have
$$
\sum_{n=1}^\infty f_n^2 < \infty.
$$
Hence, it follows from the law of iterated logarithm for martingales
due to Stout \cite{Stout1973}, see also Corollary 6.4.25 in \cite{Duflo1997}, that for any vector $u$ of $\dR^2$,
\begin{eqnarray}
 \limsup_{n \rightarrow \infty} \Bigl(\frac{1}{2 n^3 \log \log n}\Bigr)^{1/2} M_n(u)  & = & 
 -\liminf_{n \rightarrow \infty} \Bigl(\frac{1}{2 n^3 \log \log n}\Bigr)^{1/2} M_n(u) \nonumber \\
 & = & \frac{1}{\sqrt{12}}||u|| \hspace{1cm} \text{a.s.}
 \label{PLILMN1}
\end{eqnarray}
Therefore, we obtain from \eqref{PLILMN1} that for any vector $u$ of $\dR^2$,
\begin{equation*}
 \limsup_{n \rightarrow \infty} \frac{1}{2 n^3 \log \log n} \langle u, M_n \rangle^2  = \frac{1}{12}||u||^2 \hspace{1cm} \text{a.s.}
\end{equation*}
which leads to 
\begin{equation}
 \limsup_{n \rightarrow \infty} \frac{1}{2 n^3 \log \log n} ||M_n||^2  = \frac{1}{6} \hspace{1cm} \text{a.s.}
 \label{PLILMN2}
\end{equation}
Finally, we deduce \eqref{LILTRACEDN} from \eqref{PLILMN2}, which achieves the proof of Propostion \ref{P-QSLLIL}.
\end{proof}

\bibliographystyle{abbrv}

\end{document}